\documentclass[12pt]{amsart}
\usepackage{lineno}
 \usepackage{amscd,amsmath,amsthm,amssymb}
 \usepackage{pstcol,pst-plot,pst-3d}
 \psset{unit=0.7cm,linewidth=0.8pt,arrowsize=2.5pt 4}

\newpsstyle{fatline}{linewidth=1.5pt}
\newpsstyle{fyp}{fillstyle=solid,fillcolor=verylight}
\definecolor{verylight}{gray}{0.97}
\definecolor{light}{gray}{0.93}
\definecolor{medium}{gray}{0.82}
 \def\NZQ{\Bbb}               

 \def\CC{{\NZQ C}}

 \def\DD{{\NZQ D}}
 \def\FF{{\NZQ F}}
 \def\GG{{\NZQ G}}

 \def\C{{\mathcal C}}

 \def\P{{\mathcal P}}

 \def\ab{{\bold a}}
 \def\bb{{\bold b}}
 \def\xb{{\bold x}}

 \def\cb{{\bold c}}
 \def\db{{\bold d}}

 \def\opn#1#2{\def#1{\operatorname{#2}}} 
 \opn\chara{char} \opn\length{\ell} \opn\pd{pd} \opn\rk{rk}
 \opn\projdim{proj\,dim} \opn\injdim{inj\,dim} \opn\rank{rank}
 \opn\depth{depth} \opn\grade{grade} \opn\height{height}
 \opn\embdim{emb\,dim} \opn\codim{codim}
 
 \opn\Tr{Tr} \opn\bigrank{big\,rank}
 \opn\superheight{superheight}\opn\lcm{lcm}
 \opn\trdeg{tr\,deg}
 \opn\reg{reg} \opn\lreg{lreg} \opn\ini{in} \opn\lpd{lpd}
 \opn\size{size} \opn\sdepth{sdepth}
 \opn\link{link}\opn\fdepth{fdepth}\opn\lex{lex}
 \opn\div{div} \opn\Div{Div} \opn\cl{cl} \opn\Cl{Cl}
 \opn\Spec{Spec} \opn\Supp{Supp} \opn\supp{supp} \opn\Sing{Sing}
 \opn\Ass{Ass} \opn\Min{Min}\opn\Mon{Mon}
 \opn\Ann{Ann} \opn\Rad{Rad} \opn\Soc{Soc}
 \opn\Im{Im} \opn\Ker{Ker} \opn\Coker{Coker} \opn\Am{Am}
 \opn\Hom{Hom} \opn\Tor{Tor} \opn\Ext{Ext} \opn\End{End}
 \opn\Aut{Aut} \opn\id{id}
 
 \opn\nat{nat}
 \opn\pff{pf}
 \opn\Pf{Pf} \opn\GL{GL} \opn\SL{SL} \opn\mod{mod} \opn\ord{ord}
 \opn\Gin{Gin} \opn\Hilb{Hilb}\opn\sort{sort}
 \opn\aff{aff} \opn
 \con{conv} \opn\relint{relint} \opn\st{st}
 \opn\lk{lk} \opn\cn{cn} \opn\core{core} \opn\vol{vol}
 \opn\link{link} \opn\star{star}\opn\lex{lex}\opn\set{set}
 \opn\gr{gr}
 
 \def\pot#1#2{#1[\kern-0.28ex[#2]\kern-0.28ex]}

 \opn\dirlim{\underrightarrow{\lim}}
 \opn\inivlim{\underleftarrow{\lim}}
 \let\union=\cup
 \let\sect=\cap

 \let\iso=\cong

 \let\Dirsum=\bigoplus
 
 \let\to=\rightarrow
 \let\To=\longrightarrow
 \def\Implies{\ifmmode\Longrightarrow \else
         \unskip${}\Longrightarrow{}$\ignorespaces\fi}
 \def\implies{\ifmmode\Rightarrow \else
         \unskip${}\Rightarrow{}$\ignorespaces\fi}
 \def\iff{\ifmmode\Longleftrightarrow \else
         \unskip${}\Longleftrightarrow{}$\ignorespaces\fi}

 \let\:=\colon
 \newtheorem{Theorem}{Theorem}[section]
 \newtheorem{Lemma}[Theorem]{Lemma}
 \newtheorem{Corollary}[Theorem]{Corollary}
 \newtheorem{Proposition}[Theorem]{Proposition}
 \newtheorem{Remark}[Theorem]{Remark}

 \let\epsilon\varepsilon
 \let\kappa=\varkappa
 \def\qed{\ifhmode\textqed\fi
       \ifmmode\ifinner\quad\qedsymbol\else\dispqed\fi\fi}
 \def\textqed{\unskip\nobreak\penalty50
        \hskip2em\hbox{}\nobreak\hfil\qedsymbol
        \parfillskip=0pt \finalhyphendemerits=0}
 \def\dispqed{\rlap{\qquad\qedsymbol}}
 
 \opn\dis{dis}
 \def\pnt{{\raise0.5mm\hbox{\large\bf.}}}
 \def\lpnt{{\hbox{\large\bf.}}}
 \opn\Lex{Lex}

\begin{document}

 \title{Expansions of monomial ideals and multigraded modules}

  \author {Shamila Bayati}

\address{Faculty of Mathematics and Computer Science\\
Amirkabir University of Technology (Tehran Polytechnic)\\ 424 Hafez Ave.\\ Tehran
15914\\ Iran
}\email{shamilabayati@gmail.com}

\author {J\"urgen Herzog}

\address{Fachbereich Mathematik\\ Universit\"at Duisburg-Essen\\ Campus Essen\\ 45117
Essen\\ Germany} \email{juergen.herzog@uni-essen.de}

 \begin{abstract}
We introduce an exact functor defined on multigraded modules which we call the expansion functor and study its homological properties. The expansion functor applied to a monomial ideal amounts to substitute the variables by monomial prime ideals and to apply this substitution to the generators of the ideal. This operation naturally occurs in various combinatorial contexts.
 \end{abstract}


 \maketitle

 \section*{Introduction}
 In this paper we first study an operator defined on monomial ideals and their multigraded  free resolution which we call  the {\em expansion operator}. The definition of this operator is motivated by constructions in various combinatorial contexts. For example, let $G$ be a finite simple graph with vertex set $V(G)=[n]$  and edge set $E(G)$, and let  $I(G)$ be its edge ideal in $S=K[x_1,\ldots,x_n]$.  We fix a vertex $j$ of $G$. Then a new graph $G'$ is defined by  duplicating $j$, that is, $V(G') =V(G)\union\{j'\}$ and
 \[
 E(G')=E(G)\union\{\{i,j'\}\:\; \{i,j\}\in E(G)\}
 \]
 where $j'$ is new vertex. It follows that $I(G')=I(G)+(x_ix_{j'}\:\; \{i,j\}\in E(G))$. This duplication can be iterated. We denote by $G^{(i_1,\ldots,i_n)}$ the graph which is obtained from $G$ by $i_j$ duplications of  $j$. Then edge  ideal of  $G^{(i_1,\ldots,i_n)}$ can be described as follows: let  $S^*$ be the polynomial ring over $K$ in the variables
 \[
 x_{11},\ldots,x_{1i_1}, x_{21},\ldots,x_{2i_2},\ldots,  x_{n1},\ldots,x_{ni_n},
 \]
and  consider the monomial prime ideal $P_j=(x_{j1},\ldots,x_{ji_j})$  in $S^*$. Then
\[
I(G^{(i_1,\ldots,i_n)})= \sum_{\{v_i,v_j\}\in E(G)}P_iP_j.
\]
We say the ideal $I(G^{(i_1,\ldots,i_n)})$ is obtained from $I(G)$ by expansion with respect to the $n$-tuple  $(i_1,\ldots,i_n)$ with positive integer entries.

More generally, let $I=(\xb^{\ab_1},\ldots,\xb^{\ab_m})\subset S$ be any monomial ideal and   $(i_1,\ldots,i_n)$  an $n$-tuple with positive integer entries. Then we define the {\em expansion of $I$
with respect to $(i_1,\ldots,i_n)$}  as the monomial ideal
\[
I^*=\sum_{i=1}^m P_1^{a_i(1)}\cdots P_n^{a_i(n)}\subseteq S^*
\]
where $\ab_i=(a_i(1),\ldots , a_i(n))$.  A similar construction by a sequence of duplications of vertices of a simplicial complex is applied in \cite{HHTZ} to derive an  equivalent condition for vertex cover algebras of simplicial complexes  to be standard graded. In this case the cover ideal of the constructed simplicial complex is an expansion of the cover ideal of the original  one. There is still another instance, known to us, where expansions of monomial ideals naturally appear. Indeed, let $\mathcal{M}$ be any vector matroid on the ground set $E$  with set of bases $\mathcal{B}$, that is, $E$ is a finite subset of nonzero vectors of a vector space $V$ over a field $F$ and $\mathcal{B}$ is the set of maximal linearly independent subsets of $E$. We define on $E$ the following equivalence relation: for $v,w\in E$ we set $v\sim w$ if and only if $v$ and $w$ are linearly dependent. Next we choose  one representative of each  equivalence class, say, $v_1,\ldots,v_r$. Let $\mathcal{M}'$  be the vector  matroid   on the ground set $E'=\{v_1,\ldots,v_r\}$. Then, obviously, the matroidal ideal attached to $\mathcal{M}$ is an expansion of the matroidal ideal attached to $\mathcal{M}'$.

In the first section  we present the basic properties of the expansion operator, and show  this operator commutes  with the standard algebraic operations on ideals; see Lemma~\ref{operations} and Corollary~\ref{symbolic-power}.   It also commutes with primary decompositions of monomial ideals as shown in Proposition~\ref{primarydecom}. As a consequence one obtains that $\Ass(S^*/I^*)=\{P^*\:\; P\in \Ass(S/I)\}$. Also it is not so hard to see that  $I^*$ has linear quotients if $I$ has linear quotients. In view of this fact it is natural to ask whether $I^*$ has a linear resolution if $I$ has a linear resolution. This question has a positive answer, as it is shown in the following sections where we study the homological properties of $I^*$.

It turns out, as shown in Section 2, that the expansion operator can be made an exact  functor from  the category of finitely generated multigraded $S$-modules to the category  of finitely generated multigraded $S^*$-modules. Let $M$ be a finitely generated multigraded $S$-module. Applying this functor to a multigraded free resolution $\FF$ of $M$ we obtain an acyclic complex $\FF^*$ with $H_0(\FF^*)\iso M^*$. Unfortunately, $\FF^*$  is not a free resolution of $M^*$ because the expansion functor applied to a free module doesn't necessarily  yield a free module. To remedy this problem we  construct in Section 3, starting from $\FF^*$, a double complex which provides a minimal multigraded free resolution of $M^*$.  In the last section we use this construction to show that $M$ and  $M^*$ have the same regularity. 
 We also compute the graded Betti numbers of $M^*$ in terms of multigraded free resolution of $M$.

 \section{Definition and basic properties of the expansion operator}
In this section we define the expansion operator on monomial ideals and compare some algebraic properties of the constructed ideal with the original one.

 Let $K$ be a field and $S=K[x_1,\ldots,x_n]$  the polynomial ring over $K$ in the variables $x_1,\ldots,x_n$. Fix an ordered $n$-tuple $(i_1,\ldots,i_n)$ of positive integers, and  consider the polynomial ring $S^{(i_1,\ldots,i_n)}$ over $K$ in the variables
 $$x_{11},\ldots,x_{1i_1},x_{21},\ldots,x_{2i_2},\ldots,x_{n1},\ldots,x_{ni_n}.$$
 Let $P_j$ be the monomial prime ideal $(x_{j1},x_{j2},\ldots,x_{ji_j}) \subseteq S^{(i_1,\ldots,i_n)}$. Attached to each monomial ideal $I$ with a set of monomial generators $\{\xb^{\ab_1},\ldots,\xb^{\ab_m}\}$, we define the {\em expansion of $I$ with respect to the $n$-tuple $(i_1,\ldots,i_n)$}, denoted by $I^{(i_1,\ldots,i_n)}$, to be the monomial ideal $$I^{(i_1,\ldots,i_n)}=\sum_{i=1}^m P_1^{a_i(1)}\cdots P_n^{a_i(n)}\subseteq S^{(i_1,\ldots,i_n)}.$$
 Here $a_i(j)$ denotes the $j$-th component of the vector $\ab_i$. Throughout the rest of this paper we work with the fixed $n$-tuple $(i_1,\ldots,i_n)$. So we simply write $S^*$ and $I^*$, respectively, rather than $S^{(i_1,\ldots,i_n)}$ and  $I^{(i_1,\ldots,i_n)}$.

 For  monomials $\xb^\ab$ and $\xb^\bb$ in $S$ if $\xb^\bb|\xb^\ab$, then $P_1^{a_i(1)}\cdots P_n^{a_i(n)} \subseteq P_1^{b_i(1)}\cdots P_n^{b_i(n)}$. So the definition of $I^*$ does not depend on the choice of the set of monomial generators of $I$.

 For example, consider  $S=K[x_1,x_2,x_3]$ and the ordered $3$-tuple $(1,3,2)$. Then we have $P_1=(x_{11})$, $P_2=(x_{21},,x_{22},x_{23})$ and $P_3=(x_{31},x_{32})$. So  for the monomial ideal $I=(x_1x_2,x_3^2)$, the ideal $I^*\subseteq K[x_{11},x_{21},x_{22},x_{23},x_{31},x_{32}]$ is $P_1P_2+P_3^2$, namely
 $$I^*=(x_{11}x_{21},x_{11}x_{22},x_{11}x_{23},x_{31}^2,x_{31}x_{32},x_{32}^2).$$

 We define the $K$-algebra  homomorphism $\pi:S^*\rightarrow S$ by
 \[
 \pi(x_{ij})=x_i \; \text{ for all } i,j.
 \]

As usual, we denote by $G(I)$ the unique minimal set of monomial generators of a monomial ideal $I$.
Some basic properties of the expansion operator are given in the following lemma.
\begin{Lemma}
\label{operations}
Let $I$ and $J$ be  monomial ideals. Then
\begin{enumerate}
 \item[(i)] $f\in I^*$ if and only $\pi(f)\in I$, for all $f\in S^*$;
 \item[(ii)] $(I+J)^*=I^*+ J^*$;
 \item[(iii)] $(IJ)^*=I^* J^*$;
 \item[(iv)] $(I\cap J)^*=I^* \cap J^*$;
 \item[(v)]  $(I\:J)^*=I^*\:J^*$;
 \item[(vi)] $\sqrt{I^\ast}=(\sqrt{I})^\ast;$
 \item[(vii)] If the monomial ideal $Q$ is $P$-primary, then $Q^*$ is $P^*$-primary.
\end{enumerate}
\end{Lemma}

\begin{proof}
(i) Since $I$ and $I^*$ are monomial ideals, it is enough to show the statement holds for all monomials $u\in S^*$. So let $u\in S^*$ be a monomial. Then $u\in I^*$ if and only if for some monomial $\xb^{\ab}\in G(I)$  we have $u\in P_1^{a(1)}\cdots P_n^{a(n)}$, and this is the case if and only if $\xb^{\ab}$ divides $\pi(u)$.

(ii) is trivial.

(iii) Let $A=\{x^{\ab_1}\ldots x^{\ab_m}\}$ and $B=\{x^{\bb_1}\ldots x^{\bb_{m'}}\}$ be respectively a set of monomial generators of $I$ and $J$. The assertion follows from the fact that  $\{x^{\ab_i}x^{\bb_j}\: \; 1\leq i \leq m \text{ and } 1\leq j \leq m' \}$ is a set of monomial generators for $IJ$.

(iv) With the same notation as (ii), one has the set of monomial generators $\{\lcm(x^{\ab_i},x^{\bb_j}) \:1\leq i \leq m \text{ and } 1\leq j \leq m'\}$ of $I\cap J$ where $\lcm(x^{\ab_i},x^{\bb_j})$ is the least common multiple of $x^{\ab_i}$ and $x^{\bb_j}$. Therefore,
 $$ I^*\cap J^*= (\sum_{i=1}^m P_1^{a_i(1)}\cdots P_n^{a_i(n)})\cap(\sum_{j=1}^{m'} P_1^{b_j(1)}\cdots P_n^{b_j(n)})=$$

 $$\sum_{ 1\leq j \leq m'}\sum_{1\leq i \leq m } (P_1^{a_i(1)}\cdots P_n^{a_i(n)})\cap (P_1^{b_j(1)}\cdots P_n^{b_j(n)})= (I\cap J)^*$$
 The second equality holds because the summands are all monomial ideals.

(v) Since $I\:J= \bigcap_{u\in G(J)} I\:(u)$, by statement (iv) it is enough to show that  $(I\:(u))^*=I^*\:(u)^*$ for all monomials $u\in S$. Observe that if $u=\xb^{\ab}\in S$, then $(u)^*= P_1^{a_i(1)}\cdots P_n^{a_i(n)}\subseteq S^*$. In addition by properties of ideal quotient, one has that $I\:L_1L_2=(I:L_1):L_2$ for all ideals $L_1,L_2$ of $S$. Therefore, we only need to show that for each variable $x_j\in S$ the equality $(I\:x_j)^*=I^*\: P_j$ holds where  we set $P_j=(x_{j1},\ldots,x_{ji_j})$ as before.  So let $f\in S^*$. Then by (i), one has $fx_{j\ell}\in I^*$  for all $\ell$ if and only if $\pi(f)x_j=\pi(fx_{j\ell})\in I$. Hence $(I\:x_j)^*=I^*\: P_j$.

(vi) One should only notice that in general for any two ideals $L_1$ and $L_2$, we have $\sqrt{L_1+L_2}=\sqrt{\sqrt{L_1}+\sqrt{L_2}}$. With this observation, the definition of $I^*$ implies  (vi).

(vii) Let for two monomials $u,v\in S^*$, $uv\in Q^*$ but $u\not\in Q^*$. Then by (i) we have $\pi(uv)=\pi(u)\pi(v)\in Q$ and $\pi(u)\not\in Q$.  Hence $\pi(v)\in \sqrt{Q}=P$ because $Q$ is a primary ideal. Now the result follows from (i) and (vi).
 \end{proof}
 \begin{Proposition}
 \label{primarydecom}
 Let $I$ be a monomial ideal, and consider an (irredundant) primary decomposition $I=Q_1\cap \ldots\cap Q_m$ of $I$. Then
 $I^*=Q_1^*\cap \ldots\cap Q_m^*$ is an (irredundant)  primary decomposition of $I^*$.

 In particular, $\Ass(S^*/I^*)=\{P^*\:\; P\in \Ass(S/I)\}.$
 \end{Proposition}
\begin{proof}
The statement follows from part (iv) and (vii) of Lemma \ref{operations}.
\end{proof}

\begin{Remark}
{\em Let $I\subseteq S$ be an ideal  of height $r$ such that $\height P=r$ for all $P\in \Ass(S/I)$. Without loss of generality, we may assume that $i_1\leq i_2\leq \cdots \leq i_n$. Then Proposition~\ref{primarydecom} implies that $\height I^*=i_1+\cdots+i_r$, and consequently $\dim S^*/I^*=i_{r+1}+\cdots+i_n\geq n-r$. As a result,  for such an ideal $I$, we have $\dim S/I=\dim S^*/I^*$ if and only if $i_1=\ldots=i_n=1$ or $\dim S/I=0$.      }
\end{Remark}
By a result of Brodmann \cite{B}, in any Noetherian ring $R$ the set $\Ass(R/I^s)$  stabilizes for $s\gg 0$, that is, there exists a number $s_0$ such that $\Ass(R/I^s)=\Ass(R/I^{s_0})$ for all $s\geq s_0$. This stable set $\Ass(R/I^{s_0})$ is denoted by $\Ass^\infty(I)$
\begin{Corollary}
Let $I$ be a monomial ideal. Then $\Ass^\infty(I^*)=\{P^*\: P\in \Ass^\infty(I)\}$.
\end{Corollary}
\begin{proof}
By part (iii) of Lemma \ref{operations} one has $(I^k)^*=(I^*)^k$. Hence  $\Ass(S^*/(I^*)^k)=\{P^*\:\; P\in \Ass(S/I^k)\}$ by Proposition \ref{primarydecom}.
\end{proof}
Let $I$ be a monomial ideal with $\Min(I)=\{P_1,\ldots,P_r\}$.  Given an integer $k\geq 1$, the $k$-th symbolic power $I^{(k)}$ of $I$  is defined to be
$$I^{(k)}=Q_1\cap\ldots \cap Q_r$$
where $Q_i$ is the $P_i$-primary component of $I^k$.
\begin{Corollary}
\label{symbolic-power}
Let $I$ be a monomial ideal. Then ${(I^{(k)})}^*={(I^*)}^{(k)} $ for all $k\geq 1$.
\end{Corollary}

\begin{Proposition}
Let I be a  monomial ideal. If $I$ has linear quotients, then $I^*$ also has linear quotients.
\end{Proposition}
\begin{proof}
Let $I$ have linear quotients with respect to the ordering
\begin{eqnarray}
\label{order1}
u_1, u_2, \ldots, u_m
\end{eqnarray}
 of $G(I)$. We show that $I^*$ has linear quotients with respect to the ordering
\begin{eqnarray}
\label{order2}
u_{11},\ldots,u_{1r_1},\ldots,u_{m1},\ldots,u_{mr_m}
\end{eqnarray}
of $G(I^*)$ where for all $i,j$ we have $\pi(u_{ij})=u_i$, and $u_{i1}<_{\Lex} \cdots <_{\Lex} u_{ir_i}$ by the ordering of
$$x_{11}>\cdots >x_{1i_1}>\cdots>x_{n1}>\cdots >x_{ni_n}.$$
Let $v,w\in G(I^*)$ be two monomials such that in (\ref{order2}) the monomial $v$ appears before $w$. In order to prove that $I^*$ has linear quotients with respect to the above mentioned order, we must  show that there exists a variable $x_{ij}$ and a monomial $v'\in G(I^*)$ such that  $x_{ij}|(v/\gcd(v,w))$, in (\ref{order2}) the monomial $v'$ comes before $w$,  and $x_{ij}=v'/\gcd(v',w)$.

First assume that $\pi(v)=\pi(w)$, and $x_{ij}$ is the greatest variable with respect to the given order on the variables such that $x_{ij}|(v/\gcd(v,w))$. Let $v'\in G(I^*)$ be the monomial with $\pi(v')=\pi(w)$ and $v'=x_{x_{ij}}\gcd(w,v')$. By the choice of $x_{ij}$ and since the order of  monomials in (\ref{order2}) are given by the lexicographic order, the monomial  $v'$ comes before $w$ in~(\ref{order2}), as desired.

Next assume that $\pi(v)\neq\pi(w)$. Since $I$ has linear quotients, there exists a monomial $u\in G(I)$, coming before $\pi(w)$ in~(\ref{order1}), for which             there exists a variable $x_{i}$ such that $x_i|(\pi(v)/\gcd(\pi(v),\pi(w))$ and $x_i=u/\gcd(u,\pi(w))$. Let $x_{ij}$ be the variable which divides $v/\gcd(v,w)$.     Then there exists a monomial $v'\in G(I^*)$ with $\pi(v')=u$, and $v'=x_{ij}\gcd(v',w)$. Since $\pi(v')$ comes before $\pi(w)$ in~(\ref{order1}), the monomial $v'$ also appears before $w$ in~(\ref{order2}). Thus $v'$ is the desired element.
\end{proof}

\section{The expansion functor}
Throughout this paper, we consider the  standard  multigraded structure for polynomial rings, that is, the  graded components are the one-dimensional $K$-spaces spanned by the monomials. Let $\C(S)$ denote the category of finitely generated multigraded $S$-modules whose morphisms are
the multigraded $S$-module homomorphisms.   We fix an $n$-tuple $(i_1,\ldots,i_n)$ and define the expansion functor which assigns a multigraded module $M^*$ over the standard multigraded polynomial ring $S^*$ to each $M\in \C(S)$.

In the first step we define the expansion functor on multigraded free modules and multigraded maps between them.


\medskip

Let $F=\bigoplus_i S(-\ab_i)$ be a multigraded free $S$-module. Then $F$, as a multigraded module, is isomorphic to  the direct sum $\bigoplus_i(\xb^{\ab_i})$ of principal monomial ideals of $S$. This presentation of $F$ by principal monomial ideals is unique. We define the $S^*$-module $F^*$ to be
\[
F^*=\bigoplus_i(\xb^{\ab_i})^*.
\]
Observe that in general $F^*$ is no longer a free $S^*$-module.
Let $$\alpha\:\bigoplus_{i=1}^r S(-\ab_i)\rightarrow \bigoplus_{j=1}^s S(-\bb_j)$$ be a multigraded $S$-module homomorphism. Then for  the restriction $\alpha_{i,j}\: S(-\ab_i)\rightarrow  S(-\bb_j)$ of $\alpha$ we have
\[
\alpha_{i,j}(f)=\lambda_{ji}\xb^{\ab_i-\bb_j}f \;\;\; \text{for all $f\in S(-\ab_i)$}
\]
with  $ \lambda_{ji}\in K$,  and  $\lambda_{ji}=0$ if $\xb^{\bb_j}$ does not divide $\xb^{\ab_i}$. The  $(s\times r)$-matrix $[\lambda_{ji}]$  is called the {\em monomial matrix expression} of  $\alpha$.

Notice that if  $\xb^{\bb_j}$  divides  $\xb^{\ab_i}$, then $(\xb^{\ab_i})^* \subseteq (\xb^{\bb_j})^*$. So we define a multigraded $S^*$-module homomorphism $\alpha^*\:\bigoplus_{i=1}^r (\xb^{\ab_i})^*\rightarrow \bigoplus_{j=1}^s(\xb^{\bb_j})^*$ associated to $\alpha$ whose restriction  $\alpha_{ij}^*\: (\xb^{\ab_i})^*\rightarrow (\xb^{\bb_j})^*$  is defined as follows: if $\xb^{\bb_j}| \xb^{\ab_i}$,   then
$$
 \alpha_{ij}^*(f)=\lambda_{ji}f \;\;\;\;\text{for all $f\in (\xb^{\ab_i})^*$},
$$
and $\alpha_{ij}^*$ is the zero map if  $\xb^{\bb_j}$ does not divide  $\xb^{\ab_i}$.

Obviously one has $(\id_{F})^*=\id_{F^*}$ for each finitely generated multigraded free  $S$-module $F$.

\begin{Lemma}
\label{complex+exact}
Let
\begin{eqnarray}
\label{seq}
\bigoplus_{i=1}^r S(-\ab_i)\overset{\alpha}{\to} \bigoplus_{j=1}^s S(-\bb_j) \overset{\beta}{\to} \bigoplus_{k=1}^t S(-\cb_k)
\end{eqnarray}
 be  a sequence of multigraded $S$-module homomorphisms. Then the following statements hold:
\begin{enumerate}
\item[(i)] $(\beta\alpha)^*={\beta}^*\alpha^*$;
\item[(ii)] If the sequence {\em (\ref{seq})} is exact, then
\begin{eqnarray}
\label{star-seq}
\bigoplus_{i=1}^r(\xb^{\ab_i})^* \overset{\alpha^*}{\to} \bigoplus_{j=1}^s(\xb^{\bb_j})^* \overset{\beta^*}{\to}  \bigoplus_{k=1}^t(\xb^{\cb_k})^*
\end{eqnarray}
is also an exact sequence of multigraded $S^*$-modules.
\end{enumerate}
\end{Lemma}
\begin{proof}
Let $[\lambda_{ji}]$ and  $[\mu_{kj}]$ be respectively monomial matrix expressions of $\alpha$ and $\beta$.

(i) follows from the definition and the  fact that the product $[\mu_{kj}][\lambda_{ji}]$ is the monomial expression of $\beta\alpha$.

(ii) By abuse of  notation, we may consider  the following sequence which is isomorphic to the sequence (\ref{seq}):
\[
\bigoplus_{i=1}^r(\xb^{\ab_i}) \overset{\alpha}{\to} \bigoplus_{j=1}^s(\xb^{\bb_j}) \overset{\beta}{\to}  \bigoplus_{k=1}^t(\xb^{\cb_k}).
\]
Here the maps  are given by the matrices $[\lambda_{ji}]$ and  $[\mu_{kj}]$.
In order to show that the sequence~(\ref{star-seq}) is exact, it is enough to show that it is exact in each multidegree.
For this purpose, first consider the following commutative diagram

\begin{eqnarray}
\label{diagram}
\begin{CD}
\bigoplus_{i=1}^r(\xb^{\ab_i})^* @>\alpha^*>> \bigoplus_{j=1}^s(\xb^{\bb_j})^* @>\beta^*>>  \bigoplus_{k=1}^t(\xb^{\cb_k})^*\\
@V\overline{\pi} VV  @V\overline{\pi} VV @V\overline{\pi} VV \\
\bigoplus_{i=1}^r(\xb^{\ab_i}) @>\alpha>> \bigoplus_{j=1}^s(\xb^{\bb_j}) @>\beta>>  \bigoplus_{k=1}^t(\xb^{\cb_k})
\end{CD}
\vspace{0.1cm}
\end{eqnarray}
where the vertical maps $\overline{\pi}$ are defined as follows: if $J_1,\ldots,J_m$ are monomial ideals of $S$, then
$\overline{\pi}:J_1^*\oplus\cdots\oplus J_m^*\rightarrow J_1\oplus\cdots\oplus J_m$ is defined by
\[
\overline{\pi}(f_1,\ldots,f_m)=(\pi(f_1),\ldots,\pi(f_m));
\]
here $\pi$ is the $K$-algebra homomorphism introduced in the first section.
Let $\eta=\sum_{j=1}^n i_j$. Consider~(\ref{diagram}) in each multidegree $\db^*=(d_{1,1},\ldots,d_{1,i_1},\ldots,d_{n,1},\ldots,d_{n,i_n})\in \mathbb{N}^\eta$ and let $\db=(d_1,\ldots,d_n)$ where $d_j=d_{j,1}+\cdots+d_{j,i_j}$ for all $j$, namely
\begin{eqnarray}
\label{diagram2}
\begin{CD}
(\bigoplus_{i=1}^r(\xb^{\ab_i})^*)_{\db^*} @>\alpha^*>> (\bigoplus_{j=1}^s(\xb^{\bb_j})^*)_{\db^*} @>\beta^*>>  (\bigoplus_{k=1}^t(\xb^{\cb_k})^*)_{\db^*}\\
@V\overline{\pi} V\cong V  @V\overline{\pi} V\cong V @V\overline{\pi} V\cong V \\
(\bigoplus_{i=1}^r(\xb^{\ab_i}))_{\db} @>\alpha>> (\bigoplus_{j=1}^s(\xb^{\bb_j}))_{\db} @>\beta>>  (\bigoplus_{k=1}^t(\xb^{\cb_k}))_{\db}
\end{CD}
\vspace{0.2cm}
\end{eqnarray}
The above diagram is commutative. Furthermore, the restriction of $\pi:S^*\to S$ to the multidegree $d^*$, namely $\pi:S^*_{d^*}\to S_d$, is an isomorphism of $K$-vector spaces. Therefore, the given restrictions of $\overline{\pi}$ in~(\ref{diagram2}) are isomorphism of $K$-vector spaces. Hence the exactness of the second row in~(\ref{diagram2}) implies the exactness of the first row.
\end{proof}

Now we are ready to define the expansion functor $\C(S)\to \C(S^*)$. For each $M\in \C(S)$ we choose a multigraded free presentation
\[
F_1\overset{\varphi}{\To} F_0 \to M\to 0,
\]
and set $M^*=\Coker \varphi^*$. Let $\alpha:M\to N$ be a morphism in  the category $\C(S)$, and let $F_1\overset{\varphi}{\to} F_0 \to M\to 0$ and $G_1\overset{\psi}{\to} G_0 \to N\to 0$ be respectively a multigraded free presentation of $M$ and $N$.
We choose a lifting of $\alpha$:
\[
\begin{CD}
F_1      @>\varphi>>    F_0           @>>>        M         @>>>     0\\
@V\alpha_1VV            @V\alpha_0VV         @VV\alpha V\\
G_1      @>\psi>>    G_0           @>>>        N         @>>>     0
\end{CD}
\vspace{0.2cm}
\]
Let $\alpha^*$ be the map induced by maps $\alpha_0^*$ and $\alpha_1^*$. It is easily seen this assignment defines an additive functor and as a consequence of Lemma~\ref{complex+exact},  such defined expansion functor is exact. In particular we have

\begin{Theorem}
\label{exact}
Let $M$ be a finitely generated multigraded $S$-module with a multigraded free resolution
\[
\FF\: 0\to F_p \overset{\varphi_p}{\to} F_{p-1}\to \cdots \to F_1 \overset{\varphi_1}{\to} F_0\to 0.
\]
Then
\[
\FF^*\: 0\to F_p^*\overset{\varphi_p^*}{\to} F_{p-1}^*\to \cdots \to F_1^*\overset{\varphi_1^*}{\to} F_0^*\to 0
\]
is an acyclic sequence of multigraded $S^*$-modules   with $H_0(\FF^*)=M^*$.
\end{Theorem}

Let $I\subseteq S$ be a monomial ideal. The monomial ideal $I^*$ as defined in Section~1 is isomorphic, as a multigraded $S^*$-module, to the module which is obtained by applying the expansion functor to $I$. Therefore, there is no ambiguity in our notation. Indeed, applying the expansion functor to the exact sequence of $S$-modules
\[
0\to I \hookrightarrow S \to S/I \to 0,
\]
and using the fact that it is an exact functor we obtain that $(S/I)^*\iso S^*/I^*$. On the other hand, in order to compute  $(S/I)^*$ we choose a multigraded free presentation
\[
F_1\overset{\varphi_1}{\To} S\To S/I\To 0.
\]
Then $(S/I)^*\iso \Coker(\varphi_1^*)$. Since the image of $\varphi_1^*$  coincides with the ideal  $I^*$ as defined in Section~1, the assertion follows.

\section{The free $S^*$-resolution of $M^*$}
Let $M$ be a finitely generated multigraded $S$-module, and  $\FF$ be a  multigraded free resolution of $M$. We are going to construct a multigraded free resolution of $M^*$, the expansion of $M$ with respect to the $n$-tuple $(i_1,\ldots,i_n)$, by  the resolution $\FF$. To have a better perspective on this construction, let
\[
\FF\: 0\to F_p \overset{\varphi_p}{\to} F_{p-1}\to \cdots \to F_1 \overset{\varphi_1}{\to} F_0\to 0,
\]
and consider the acyclic complex $\FF^*$, see Theorem \ref{exact},
\[
\FF^*\: 0\to F_p^*\overset{\varphi_p^*}{\to} F_{p-1}^*\to \cdots \to F_1^*\overset{\varphi_1^*}{\to} F_0^*\to 0.
\]
We will first construct a minimal multigraded free resolution $\GG_i$ of each $F_i^*$, and  a natural lifting $\GG_i\to \GG_{i-1}$ of each map $\varphi_i^*$ to obtain a double complex $\CC$. Then we show that the total complex of $\CC$ is a free resolution of $M^*$,  and it is minimal if the free resolution $\FF$ is  minimal.

\medskip
To obtain the desired resolution $\GG_i$ of $F_i^*$ and a suitable lifting of $\varphi_i^*$ we need some preparation.
Let $J\subseteq S=K[x_1,\ldots,x_n]$ be a monomial ideal for which $G(J)=\{u_1,\ldots,u_m\}$ is the minimal set of monomial generators with a non-decreasing ordering with respect to their total degree.  Suppose that $J$ has linear quotients with respect to the ordering $u_1,\ldots,u_m$ of generators. Then $\set(u_j)$ is defined to be
$$\set(u_j)=\{k\in [n]\: x_k\in (u_1,\ldots,u_{j-1}):u_j\} \; \text{ for all $j=1,\ldots,m$}.$$
In \cite[Lemma 1.5]{HT} a minimal multigraded free resolution $\GG(J)$ of $J$ is given  as follows:  the $S$-module $G_i(J)$ in  homological degree $i$ of $\GG(J)$ is the multigraded free $S$-module whose basis is formed by the symbols
$$f(\sigma;u)\; \text{ with } u\in G(J),\; \sigma\subseteq\set(u), \text{ and } |\sigma|=i.$$
 Here $\deg f(\sigma;u)=\mathbf{\sigma}+\deg u$ where $\sigma$ is identified with the $(0,1)$-vector in $ \mathbb{N}^n$ whose $k$-th component is $1$ if and only if $k\in \sigma$. The augmentation map $\varepsilon:G_0\to J$ is defined by $\varepsilon(f(\emptyset;u))=u$.

 In \cite{HT} the chain map of $\GG(J)$ for a class of such ideals is  described. We first recall some definitions needed to present this chain map. Let $M(J)$ be the set of all monomials belong to $J$. The map $g\: M(J)\to G(J)$ is defined as follows:  $g(u) = u_j$
if $j$ is the smallest number such that $u \in (u_1,\ldots,u_j)$. This map is called the {\em decomposition function of $J$ }. The {\em complementary factor} $c(u)$ is defined to be the monomial for which $u=c(u)g(u)$.

The decomposition function $g$ is called {\em regular} if $\set(g(x_su))\subseteq \set(u)$ for all $s\in \set(u)$ and $u\in G(J)$. By \cite[Theorem 1.12]{HT}  if the decomposition function of $J$ is regular, then the chain map $\partial$ of $\GG(J)$ is given by
\[
\partial(f(\sigma;u))= -\sum_{t\in \sigma} (-1)^{\alpha(\sigma,t)}x_t f(\sigma\setminus t;u)
+\sum_{t\in \sigma}(-1)^{\alpha(\sigma,t)}\frac{x_t u}{g(x_t u)}f(\sigma\setminus t;g(x_t u)).
\]
Here the definition  is extended by setting $f(\sigma;u)=0$ if  $\sigma \not\subseteq \set(u)$, and
$$\alpha(\sigma,t)=|\{s\in \sigma\: s<t\}|.$$

Now let  $P\subseteq S$ be a monomial prime ideal and $a$ be a positive integer. Considering the lexicographic order on the minimal set of monomial generators  of $P^a$ by the ordering $x_1>\cdots>x_n$, the ideal $P^a$ has linear quotients. Moreover, $\set(u)=\{1,\ldots,m(u)-1\}$ for all $u$ in the minimal set of monomial generators of $P^a$, where $$m(u)=\max\{i\in [n]\: x_i| u\}.$$ We denote the decomposition function of $P^a$ by $g_a$ and the complementary factor of each monomial $u$ by $c_a(u)$. The decomposition function  $g_a$ is regular. Hence one can apply the above result to find the chain map of the minimal free resolution of $P^a$.
We denote by $(\GG(P^a), \partial)$ the resolution of the $a$-th power of a monomial prime ideal $P$ obtained in this way.

\medskip
Let $u\in P^a$ be a monomial. We may write $u$ uniquely in the form $u=v\prod_{k=1}^dx_{j_k}$ such that $v\not\in P$,   $x_{j_k}\in P$ for all $k$, and $j_1\leq j_2\leq \cdots \leq j_d$.  Observe that with the above mentioned ordering on the minimal set of monomial generators of $P^a$ we obtain
\begin{eqnarray}
\label{formula}
g_a(u)=\prod_{k=1}^a x_{j_k}.
\end{eqnarray}

\begin{Lemma}
\label{dec.map}
Let $P$ be a monomial prime ideal and $u$ be a monomial in the minimal set of monomial generators of $P^{a}$. If $a\geq b$, then
$$g_{b}(g_{a}(x_tu))=g_{b}(x_tg_{b}(u))\;\; \text{ for all $t$}.$$
\end{Lemma}
\begin{proof}
Let $u=\prod_{k=1}^a x_{j_k}$ with $j_1\leq \cdots \leq j_a$. Then using  formula (\ref{formula}) for $g_a$ and $g_b$ one has $g_{b}(g_{a}(x_tu))$ and $g_{b}(x_tg_{b}(u))$ are both equal to $x_t(\prod_{k=1}^{b-1} x_{j_k})$ if  $t\leq j_b$. Otherwise they are both equal to  $\prod_{k=1}^b x_{j_k}$.
\end{proof}

Let $P\subseteq S$ be a monomial prime ideal and consider  the resolutions $\GG(P^a)$ and $\GG(P^b)$. Suppose that $a\geq b$. For each $s$ we define the map $\varphi_s^{a,b}\: G_s(P^a)\to G_s(P^b)$, between the modules in homological degree $s$ of $\GG(P^a)$ and $\GG(P^b)$,  to be $$\varphi_s^{a,b}(f(\sigma;u))= c_b(u)f(\sigma;g_b(u)).$$
Here, as before, we set $f(\sigma;g_b(u))=0$ if $\sigma \not\subseteq \set(g_b(u))$.
We have the following result:

\begin{Proposition}
\label{lifting}
 The map $\varphi^{a, b}=(\varphi_s^{a,b})\: \GG(P^a)\to \GG(P^b)$ is a lifting of the inclusion $S$-homomorphism $\iota\: P^{a} \to P^{b}$, that is, a complex homomorphism with $H_0(\varphi^{a,b})=\iota$.
\end{Proposition}
\begin{proof}
In this proof, for simplicity, we denote $\varphi_s^{a,b}$ by $\varphi_s$.  First observe that the following diagram is commutative:
$$
\begin{CD}
G_0(P^a) @>\varepsilon>> P^a\\
@V\varphi_0 VV  @VV\iota V \\
G_0(P^b) @>\varepsilon>> P^b
\end{CD}
$$
In fact for each element $f(\emptyset;u)$ in the basis of $G_0(P^a)$ we have
$$\varepsilon\varphi_0(f(\emptyset;u))=\varepsilon(c_b(u)f(\emptyset;g_b(u)))=c_b(u)g_b(u)=u, $$
 and also $\iota\varepsilon(f(\emptyset;u))=u$. Therefore, the above diagram is commutative.

Next we show that the following diagram is commutative for all $s\geq 1$:
$$
\begin{CD}
G_s(P^a) @>\partial>> G_{s-1}(P^a)\\
@V\varphi_s VV  @VV\varphi_{s-1} V \\
G_s(P^b) @>\partial>> G_{s-1}(P^b)
\end{CD}
$$
For each element $f(\sigma;u)$ in the basis of $G_s(P^a)$ one has
\begin{eqnarray*}
\varphi_{s-1}(\partial(f(\sigma;u)))
&=&\varphi_{s-1}(-\sum_{t\in \sigma} (-1)^{\alpha(\sigma,t)}x_t f(\sigma\setminus t;u)\\
&&\; \;\;\;\;\;\;\;+\sum_{t\in \sigma}(-1)^{\alpha(\sigma,t)}\frac{x_t u}{g_a(x_t u)}f(\sigma\setminus t;g_a(x_t u)))\\
&= &-\sum_{t\in \sigma}(-1)^{\alpha(\sigma,t)}x_t c_b(u) f(\sigma\setminus t;g_b(u))\\
&&+\sum_{t\in \sigma}(-1)^{\alpha(\sigma,t)}\frac{x_t u}{g_a(x_t u)} c_b(g_a(x_t u)) f(\sigma\setminus t;g_b(g_a(x_t u))),
\end{eqnarray*}
and on the other hand
\begin{eqnarray*}
\partial(\varphi_s(f(\sigma;u)))&= &\partial(c_b(u)f(\sigma; g_b(u)))=\\\
&-&\sum_{t\in \sigma}(-1)^{\alpha(\sigma,t)}x_t c_b(u) f(\sigma\setminus t;g_b(u))\\
&+&\sum_{t\in \sigma}(-1)^{\alpha(\sigma,t)}c_b(u)\frac{x_tg_b(u)}{g_b(x_tg_b(u))} f(\sigma\setminus t;g_b(x_tg_b(u)).
\end{eqnarray*}
\\
By Lemma \ref{dec.map} we have $$f(\sigma\setminus t;g_b(g_a(x_t u)))=f(\sigma\setminus t;g_b(x_tg_b(u)).$$
Hence we only need to show that
$$\frac{ u c_b(g_a(x_t u))}{g_a(x_t u)}=\frac{c_b(u)g_b(u)}{g_b(x_tg_b(u))}.$$
Using the facts $u=c_b(u)g_b(u)$ and $g_a(x_t u)=c_b(g_a(x_t u))g_b(g_a(x_t u))$, one has
\begin{eqnarray*}
\frac{ u\; c_b(g_a(x_t u))}{g_a(x_t u)}=\frac{c_b(u)g_b(u)\;c_b(g_a(x_t u))}{c_b(g_a(x_t u)) \; g_b(g_a(x_t u))}
=\frac{c_b(u)g_b(u)}{g_b(g_a(x_t u))}
=\frac{c_b(u)g_b(u)}{g_b(x_t g_b(u))},
\end{eqnarray*}
where the last equality is obtained by Lemma \ref{dec.map}.
\end{proof}

\begin{Proposition}
\label{mostimportant}
Let $a\geq b\geq c$ be nonnegative integers. Then $\varphi^{a,c}= \varphi^{b,c}\circ \varphi^{a,b}$.
\end{Proposition}
\begin{proof}
We must show that for all $s$ if $f(\sigma;u)$ is an  element in the basis of $G_s(P^a)$, then
$$c_c(u)f(\sigma;g_c(u))= c_b(u)\;c_c(g_b(u))f(\sigma;g_c(g_b(u))).$$
First observe that by formula given in (\ref{formula})
$$g_c(g_b(u))=g_c(u).$$
Moreover,
$$c_b(u)\;c_c(g_b(u))=c_b(u)\frac{g_b(u)}{g_c(g_b(u))}=\frac{u}{g_c(g_b(u))}=\frac{u}{g_c(u)}=c_c(u).$$
\end{proof}

Let $\xb^{\ab}\in S=K[x_1,\ldots,x_n]$ be a monomial. Recall that the expansion of ideals are given with respect to the fixed $n$-tuple $(i_1,\ldots,i_n)$, and as before we set $P_j=(x_{j1},x_{j2},\ldots,x_{ji_j})\subseteq S^*$.  We define the complex $\GG^{\ab}$ to be  $$\GG^{\ab}=\bigotimes_{j=1}^n \GG(P_j^{a(j)}).$$

\begin{Proposition}
\label{minimality2}
The complex $\GG^{\ab}$ is a minimal free resolution of $(\xb^{\ab})^*$.
\end{Proposition}
\begin{proof}
The ideal $(\xb^{\ab})^*$ is the product of the ideals $P_j^{a(j)}$. Since the minimal generators of the ideals $P_j$ are in pairwise distinct sets of variables, the assertion follows from the next simple fact.
\end{proof}

\begin{Lemma}
\label{product}
Let  $J_1,\ldots,J_r$  be  {\em (}multi{\em )}graded ideals of a polynomial ring $S$ with the property that  $$\Tor_i^S(S/(J_1\cdots J_{k-1}),S/J_{k})=0\quad  \text{for all $k=2,\ldots,r$ and  all $i>0$}.$$
Let  $\GG_k$ be a {\em (}multi{\em )}graded minimal free resolution of $J_k$ for $k=1,\ldots,r$. Then $\bigotimes_{k=1}^r\GG_k$ is a  {\em (}multi{\em )}graded minimal free resolution of $J_1\cdots J_k$.
\end{Lemma}
\begin{proof}
Let $I$ and $J$ be graded ideals of $S$ with $\Tor_i^S(S/I,S/J)=0$ for all $i>0$, and let $\FF$ and $\GG$ be respectively graded minimal free resolutions of $S/I$ and $S/J$. Suppose that $\widetilde{\FF}$ is the minimal free resolution of $I$ obtained by $\FF$, that is, we have $\widetilde{F_i}=F_{i+1}$ in each homological degree $i\geq 0$ of $\widetilde{\FF}$. In the same way, we consider the resolution $\widetilde{\GG}$ for $J$. We will show that $\widetilde{\FF}\otimes \widetilde{\GG}$ is a graded minimal free resolution of $IJ$. Then by induction on the number of ideals $J_1,\ldots,J_r$ we conclude the desired statement.

First observe that $F_0\cong S\cong G_0$ because $\FF$ and $\GG$ are  graded minimal free resolutions of $S/I$ and $S/J$. Now let $\DD$ be the subcomplex of $\FF\otimes \GG$ whose $k$-th component is
$F_0\otimes G_k \oplus F_k\otimes G_0$ for all $k>0$ and its zeroth component is $F_0\otimes G_0$. Hence $\widetilde{\FF}\otimes \widetilde{\GG}\cong ((\FF\otimes \GG) /\DD)[2]$. If $\iota\: \DD\to \FF\otimes \GG$ is the natural inclusion map, then  we have the following short exact sequence of the complexes:
\begin{eqnarray}
\label{short}
0\to \DD \overset{\iota}{\to} \FF\otimes \GG \to (\widetilde{\FF}\otimes \widetilde{\GG})[-2] \to 0.
\end{eqnarray}

Since $\Tor_i^S(S/I,S/J)=0$ for all $i>0$, we have $H_i(\FF\otimes \GG)=0$ if $i>0$; see for example \cite[Theorem 10.22]{R} . Furthermore,  $H_i(\DD)=0$ for all $i\geq 2$. Hence  by the  long exact sequence of homology modules arising from (\ref{short})  one has $H_i(\widetilde{\FF}\otimes \widetilde{\GG})=0$ if $i\neq 0$, and $H_0(\widetilde{\FF}\otimes \widetilde{\GG})=H_2(\widetilde{\FF}\otimes \widetilde{\GG})[-2] \iso H_1(\DD)$. So we only need to show that $H_1(\DD)\cong IJ$.

The complex $F_0\otimes \GG$ may be considered as a subcomplex of $\DD$ by natural inclusion map $\iota': F_0\otimes \GG \to \DD$. Then
$\widetilde{\FF}\otimes G_0 \cong (\DD /(F_0\otimes \GG))[1]$,  and we have the following short exact sequence of complexes:
\begin{eqnarray}
\label{short2}
0\to F_0\otimes \GG \overset{\iota'}{\to} \DD \to (\widetilde{\FF}\otimes G_0)[-1] \to 0.
\end{eqnarray}

The short exact sequence (\ref{short2}) yields the following exact sequence of homology modules
\[
0\to  H_1(\DD) \to  H_1((\widetilde{\FF}\otimes G_0)[-1]) \to H_0(F_0\otimes \GG) \to  H_0(\DD) \to 0
\]
which is isomorphic to
\[
0 \to H_1(\DD) \to I \to  S/J \to S/(I+J) \to 0.
\]
Thus $H_1(\DD)\cong I\cap J$. Now $(I\sect J)/(IJ)=\Tor_1^S(S/I,S/J)=0$; see for example \cite[Proposition 10.20]{R}. Hence  $I\sect J=IJ$ and consequently $H_1(\DD)\cong IJ$, as desired.
\end{proof}

Let $\xb^{\ab},\xb^{\bb}\in S$ be monomials such that $\xb^{\bb}|\xb^{\ab}$,  that is, $b(j)\leq a(j)$ for all $j=1,\ldots,n$. The  complex homomorphisms $\varphi^{a(j), b(j)}\: \GG(P_j^{a(j)})\to \GG(P_j^{b(j)})$  induce  the complex homomorphism $\varphi^{\ab,\bb}\: \GG^{\ab} \to \GG^{\bb}$ by
$\varphi^{\ab,\bb}=\bigotimes_{j=1}^n \varphi^{a(j), b(j)}$. Then $\varphi^{\ab,\bb}$ is a lifting of the inclusion map $\iota\: (\xb^{\ab})^* \to (\xb^{\bb})^*$.

\begin{Lemma}
\label{composition}
Let $\xb^{\ab},\xb^{\bb},\xb^{\cb}\in S$ be  monomials such that $\xb^{\bb}|\xb^{\ab}$ and $\xb^{\cb}|\xb^{\bb}$.
Then $$(\varphi^{\bb,\cb})\circ (\varphi^{\ab,\bb})= \varphi^{\ab,\cb}.$$
\end{Lemma}
\begin{proof}
This is a consequence of Proposition~\ref{mostimportant}.
\end{proof}

\begin{Lemma}
\label{minimality}
Let $\xb^{\ab},\xb^{\bb}\in S$ be  monomials such that $\xb^{\bb}|\xb^{\ab}$ and for some $t\in [n]$ $b(t)<a(t)$. Then
each component of $\varphi^{\ab,\bb}$ is minimal, i.e.
$$\varphi_s^{\ab,\bb}(G_s^{\ab})\subseteq \mathfrak{m}^* G_s^{\bb}$$
for all $s$ where $\mathfrak{m}^*$ is the graded maximal ideal of $S^*$.
\end{Lemma}
\begin{proof}
Consider the number $t$ for which $b(t)<a(t)$. For each $s$ one has
\begin{eqnarray}
\label{minformula}
\varphi_s^{a(t),b(t)}(f(\sigma;u))= c_{b(t)}(u)f(\sigma;g_{b(t)}(u))
\end{eqnarray}
for  all $f(\sigma;u)$ in the basis of $G_s({P_t}^{a(t)})$. Since $b(t)<a(t)$, formula (\ref{formula}) implies that $g_{b(t)}(u)\neq u$ in (\ref{minformula}). Hence $c_{b(t)}(u)\in \mathfrak{m}^*$, and consequently
$$\varphi_s^{a(t),b(t)}(f(\sigma;u))\in  \mathfrak{m}^* G_s({P_t}^{b(t)}).$$
Now the desired result follows from the definition of  $\GG^{\ab}$, $\GG^{\bb}$, and $\varphi^{\ab,\bb}$.
\end{proof}

\medskip
Let $M$ be a finitely generated multigraded $S$-module, and  suppose that
\[
\FF\: 0\to F_p \overset{\varphi_p}{\to} F_{p-1}\to \cdots \to F_1 \overset{\varphi_1}{\to} F_0\to 0
\]
is a minimal multigraded free resolution of $M$ with $F_i=\bigoplus_j S(-\ab_{ij})$. By Theorem~\ref{exact} we obtain the acyclic complex
\[
\FF^*\: 0\to F_p^*\overset{\varphi_p^*}{\to} F_{p-1}^*\to \cdots \to F_1^*\overset{\varphi_1^*}{\to} F_0^*\to 0.
\]
Now by Proposition \ref{minimality2} the complex $\GG_i=\bigoplus_j \GG^{\ab_{ij}}$ is a minimal multigraded free resolution of   $F_i^*=\bigoplus_j (\xb^{\ab_{ij}})^*$.
For each $i$ if $[\lambda_{\ell k}]$ is the monomial expression of $\varphi_i$, then we set $d:\bigoplus_j \GG^{\ab_{ij}} \to \bigoplus_j \GG^{\ab_{(i-1)j}}$ to be the complex homomorphism whose restriction to $\GG^{\ab_{ik}}$ and $\GG^{\ab_{(i-1),\ell}}$ is $\lambda_{\ell k}\varphi^{\ab_{ik},\ab_{(i-1),\ell}}$. Thus $d$ is a lifting of $\varphi_i^*$ and moreover by  Lemma~\ref{composition} we have a double complex $\CC$ with $C_{ij}=G_{ij}$, namely the module in homological degree $j$ of the complex $\GG_i$, whose column differentials are given by the differentials $\partial$ of each $\GG_i$ and row differentials are given by $d$.

\medskip

\[
\begin{CD}
0 \to @. F_p^*  @>\varphi_p^*>>  F_{p-1}^*     @. \;\to \ldots \to\;  @.   F_1^*   @>\varphi_1^*>> F_0^*     @. \to 0\\
@.       @AAA                    @AAA                 @.              @AAA                  @AAA       @.  \\
      @. G_{p0}   @>d>>            G_{p-1\;0}   @. \;\to \ldots \to\;  @.   G_{10}   @>d>>            G_{00}   @. \\
@.       @AA\partial A                    @AA\partial A                 @.              @AA\partial A                  @AA\partial A       @.  \\
      @. G_{p1}   @>d>>            G_{p-1\;1}   @. \;\to \ldots \to\;  @.   G_{11}   @>d>>            G_{01}   @. \\
@.       @AA\partial A                    @AA\partial A                 @.              @AA\partial A                  @AA\partial A       @.  \\
      @. G_{p2}   @>d>>            G_{p-1\;2}   @. \;\to \ldots \to\;  @.   G_{12}   @>d>>            G_{02}   @. \\
@.       @AA\partial A                    @AA\partial A                 @.              @AA\partial A                  @AA\partial A       @.  \\
\end{CD}
\]
\medskip
\medskip

We have the following result:
\begin{Theorem}
\label{maybe2}
The total complex $T(\CC)$ of  $\CC$ is a minimal multigraded free resolution of $M^*$.
\end{Theorem}
\begin{proof}
By Theorem~\ref{exact}  the following  sequence of multigraded modules is acyclic
\[
\FF^*\: 0\to F_p^*\overset{\varphi_p^*}{\to} F_{p-1}^*\to \cdots \to F_1^*\overset{\varphi_1^*}{\to} F_0^*\to 0,
\]
and $H_0(\FF^*)\cong M^*$. For each $i$ the complex $\GG_i=\bigoplus_j \GG^{\ab_{ij}}$ is a minimal multigraded free resolution of   $F_i^*$ because each $\GG^{\ab_{ij}}$ is a minimal free resolution of   $(\xb^{\ab_{ij}})^*$, as we know by Proposition~\ref{minimality2}.

We compute the spectral sequences  with respect to the column   filtration  of the double complex $\CC$. So $E_{r,s}^1=H_s^v(\CC_{r,\lpnt})=H_s(\GG_r)$. Since each $\GG_r$ is an acyclic sequence with $H_0(\GG_r)\cong F_r^*$, one has
$E_{r,s}^1=0$ if $s\neq 0$, and $E_{r,0}^1\cong F_r^*$.
 Next we consider $E_{r,s}^2=H_r^h(H_s^v(\CC))$. So $E_{r,s}^2=0$ if  $s\neq 0$, and  $E_{r,0}^2=H_r(\FF^*)$. By Theorem~\ref{exact} we have $H_r(\FF^*)=0$ for $r\neq 0$ and $H_0(\FF^*)=M^*$. Thus we see that $E_{r,s}^2=0$ for $(r,s)\neq (0,0)$ and  $E_{0,0}^2=M^*$. This implies that the total complex  $T(\CC)$ of $\CC$ is a multigraded free resolution of $M^*$; see for example \cite[Proposition 10.17]{R}.

Since each $\GG_i$ is a minimal free resolution of $F_i^*$, and moreover by   Lemma~\ref{minimality} each component of $d$ is also minimal, we conclude that the free resolution $T(\CC)$ of $M^*$ is minimal.
\end{proof}

\section{Homological properties of $M^*$}
This section concerns some homological properties of the expansion of a finitely generated multigraded $S$-module $M$  with respect to a fixed $n$-tuple $(i_1,\ldots,i_n)$.
We call $$\P_M(t)=\sum_{j=0}^p\beta_jt^j$$ the {\em Betti polynomial of   $M$} where $\beta_j$ is the $j$-th Betti number of $M$ and $p$ is its projective dimension.

First we need to have the Betti numbers of $I^*$ for the case that $I=(\xb^\ab)$ is a principal ideal, and hence $I^*$ is a product of powers of monomial prime ideals in pairwise distinct sets of variables.  We have the following fact:
\begin{Proposition}
\label{principal}
Let $I=(\xb^\ab)$ with $\ab=(a_1,\ldots,a_n)$. Then  $I^*$ has a linear resolution with
\[
\P_{I^*}(t)=\prod_{j\in \supp(\ab)}P_j(t),
\]
where
\[
P_j(t)=\sum_{i=0}^{i_j-1} {i_j+a_j-1\choose i_j-i-1}{a_j+i-1\choose i}t^i
\]
In particular, $\projdim I^*=\sum_{j\in \supp(\ab)}(i_j-1)$.
\end{Proposition}
\begin{proof}
By Lemma \ref{product} it is enough to compute the Betti numbers for powers of monomial prime ideals.
 So let $R=K[y_1,\ldots,y_r]$ be a polynomial ring and consider $J=(y_1,\ldots,y_r)^s$ for some positive integer $s$. Then Eagon-Northcatt complex resolving $J$ gives the Betti numbers; see for example~\cite{BV}. Alternatively one can use the Herzog-K\"{u}hl formula to obtain the Betti numbers \cite[Theorem 1]{HK}. In fact, first observe that  $J$ is Cohen-Macaulay and it has  linear resolution; see \cite[Theorem 3.1]{CH}. Next by Auslander-Buchsbaum formula, one has $\projdim (J)= r-1$. Therefore, by Herzog-K\"{u}hl formula one obtains that
\begin{eqnarray*}
\beta_j(J)&=&\frac {s (s+1)\cdots (s+j-1)\widehat{(s+j)}(s+j+1)  \cdots  (s+r-1)}{j! \times (r-j-1)!}\vspace{0.3cm}\\
          &=&\frac{(s+r-1)!}{(r-j-1)!(s+j)!}\times \frac{(s+j-1)!}{j!(s-1)!}\vspace{2cm}\\
          &&\\
          &=&{r+s-1\choose r-j-1}{s+j-1\choose j}.
\end{eqnarray*}
\end{proof}

\begin{Theorem}
\label{maybe}
Let $M$ be a finitely generated multigraded $S$-module  with the minimal multigraded free resolution
\[
\FF\: 0\to F_p \overset{\varphi_p}{\to} F_{p-1}\to \cdots \to F_1 \overset{\varphi_1}{\to} F_0\to 0,
\]
where $F_i=\Dirsum_j S(-\ab_{ij})$ for $i=0,\ldots,p$. Let $M^*$ be the expansion of $M$ with respect to $(i_1,i_2,\ldots,i_n)$.
Then
\[
\beta_{jk}(M^*)= \sum_{i=0}^p \beta_{j-i, k}(F_{
i}^*) \quad \text{for all $j$ and $k$,}
\]
Moreover,  $\reg(M)=\reg(M^*)$ and
\[
\projdim M^*=\max_{i,j}\{i+\sum_{k\in \supp(\ab_{ij})}(i_k-1)\}.
\]
\end{Theorem}
\begin{proof}
 By Theorem \ref{maybe2} we have
 \[
\beta_{jk}(M^*)= \sum_{i=0}^p \beta_{j-i, k}(F_{
i}^*) \quad \text{for all $j$ and $k$.}
\]

Next we show that $\reg(M)=\reg(M^*)$.
If $\ab=(a_1,\ldots,a_n)\in \mathbb{N}^n$ is a vector, we denote $\sum_{i=1}^n a_i$ by $|a|$. For each $F_i=\bigoplus_{j=1}^{\beta_i} S(-\ab_{ij})$ we choose $\ab_{i\ell_i}$ such that
$$|\ab_{i\ell_i}|=\max\{|\ab_{ij}|\: j=1,\ldots,\beta_i\}.$$
Then $\reg(M)=\max \{|\ab_{i\ell_i}|-i\: i=0,\ldots,p\}$.

On the other hand,  $\GG_i=\bigoplus_{j=1}^{\beta_i} \GG^{\ab_{ij}}$, and $\GG^{\ab_{ij}}$ is a $|\ab_{ij}|$-linear resolution of $(\xb^{\ab_{ij}})^*$. Hence if   for a fixed $j$ we set $r_j=\max\{k-j\: \beta_{jk}(M^*)\neq 0\}$, then
\begin{eqnarray*}
r_j&=&\max\{(|\ab_{i\ell_i}|+(j-i))-j\: i=0,\ldots,j\}\\
&=& \max\{|\ab_{i\ell_i}|-i\: i=0,\ldots,j\},
\end{eqnarray*}
and $\reg(M^*)=\max\{r_j\: j=0,\ldots,p\}$. This yields $\reg(M)=\reg(M^*)$.

By Theorem \ref{maybe2} the complex $T(\CC)$ is a minimal free resolution of $M^*$. Therefore,
\[
\projdim M^*=\max_{i} \{i+\projdim F_i^*:i=0,\ldots,p\}.
\]
On the other hand, since $F_i^*=\bigoplus_j (\xb^{\ab_{ij}})^*$, we conclude that  $\projdim F_i^*$ is the maximum number among the numbers $\projdim(\xb^{\ab_{ij}})^*$. Hence
\[
\projdim M^*=\max_{i,j}\{i+\projdim(\xb^{\ab_{ij}})^*\}
\]
Now by Proposition \ref{principal}

\[
\projdim M^*=\max_{i,j}\{i+\sum_{k\in \supp(\ab_{ij})}(i_k-1)\}.
\]
\end{proof}

As a result of Theorem \ref{maybe} we obtain the following:

\begin{Corollary}
\label{linear}
Let $M$ be a finitely generated multigraded $S$-module. Then $M$ has a $d$-linear resolution if and only if $M^*$ has a $d$-linear resolution;
\end{Corollary}
%

\begin{Remark}
{\em In Theorem \ref{maybe}  the Betti numbers of $M^*$ are given in terms of the  Betti numbers of $S^*$-modules $F_i^*$. Each $F_i^*$ is a direct sum of ideals of the form $(\xb^{\ab})^*$. So one can apply Proposition~\ref{principal} to obtain a more explicit formula for  Betti numbers of $M^*$.}
\end{Remark}
\medskip

Let $M$ be a finitely generated multigraded $S$-module with a minimal multigraded free resolution
\[
\FF\: 0\to F_p \overset{\varphi_p}{\to} F_{p-1}\to \cdots \to F_1 \overset{\varphi_1}{\to} F_0\to 0,
\]
where $F_i=\Dirsum_j S(-\ab_{ij})$ for $i=0,\ldots,p$. We call a shift $\ab_{ij}$ in $\FF$ an {\em extremal multigraded shift of $M$} if for all multigraded shifts $\ab_{k\ell}$  in $\FF$ with $k> i$  the monomial $\xb^{\ab_{ij}}$ does not divide $\xb^{\ab_{k\ell}}$. Then by Theorem~\ref{maybe}
\[
\projdim M^*=\max\{i+\sum_{k\in \supp(\ab_{ij})}(i_k-1)\:\; \ab_{ij} \text{ is an extremal shift of $\FF$}\}.
\]
This is not always the case that the extremal shifts of $\FF$ only appear in $F_p$. For example, consider the ideal $I=(x_1x_4x_6,x_2x_4x_6,x_3x_4x_5,x_3x_4x_6)$ and its minimal multigraded free resolution $\FF$. Here for simplicity we write monomials rather than shifts:
\begin{eqnarray*}
\FF\:0 \to (x_1x_2x_3x_4x_6) &\to& (x_1x_2x_4x_6)\oplus(x_1x_3x_4x_6)\oplus(x_2x_3x_4x_6)\oplus(x_3x_4x_5x_6)\\
                             &\to&(x_1x_4x_6)\oplus (x_2x_4x_6)\oplus (x_3x_4x_5)\oplus (x_3x_4x_6)\to0.
\end{eqnarray*}
Then $(0,0,1,1,1,1)$ corresponding to the monomial $x_3x_4x_5x_6$ is an extremal multigraded shift of $\FF$. However, if $M$ is  Cohen-Macaulay, then the extremal multigraded shifts appear only in the last module of the minimal multigraded free resolution, otherwise there exists a  free direct summand in  $i$-th syzygy module of $M$ for some $i<\projdim M$. But by a result of Dutta this cannot happen; see~\cite[Corollary 1.2]{D}.

{}
\end{document}